\documentclass[reqno,10pt]{amsart} 

\usepackage{amsmath,latexsym,amssymb, afterpage, xcolor}
 \usepackage{graphicx}
\input epsf%
\newtheorem{theorem}{Theorem}[section]

\newtheorem{lemma}{Lemma}[section]

\newtheorem{Definition}{Definition}[section]
\def\qed{{\hfill{\vrule height4pt width3pt depth2pt}}}

   \long\def\comment#1{}

\def\ad#1{\begin{aligned}#1\end{aligned}}  \def\b#1{\mathbf{#1}} \def\t#1{{\text{#1}}}
\def\a#1{\begin{align*}#1\end{align*}} \def\an#1{\begin{align}#1\end{align}} \def\t#1{\hbox{#1}}

  \numberwithin{equation}{section}
\numberwithin{table}{section} \numberwithin{figure}{section}

\def\boxit#1{\vbox{\hrule height1pt \hbox{\vrule width1pt\kern1pt
     #1\kern1pt\vrule width1pt}\hrule height1pt }}
 \def\lab#1{\boxit{\small #1}\label{#1}}   
 
     \def\lab#1{\label{#1}}  
\newcommand{\RR}{\mathbb{R}}

\begin{document} 
\long\def\comments#1{ }
\comments{ }

\title[Interpolation finite element]{An interpolated Galerkin finite element method for the Poisson equation }
\date{}

 \author {Tatyana Sorokina$^1$}
\address{Department of Mathematics, Towson University, 7800 York Road, Towson, MD 21252, USA.   tsorokina@towson.edu }
 
\thanks{$^1$ The author is partially supported by the grant from the Simons Foundation \#235411 to Tatyana Sorokina}

 \author {Shangyou Zhang}
\address{Department of Mathematical Sciences, University of Delaware,
    Newark, DE 19716, USA.  szhang@udel.edu }

\begin{abstract}
When solving the Poisson equation by the finite element method,
    we use one degree of freedom for interpolation by the given Laplacian --
   the right hand side function in the partial differential equation.
The finite element solution is  the Galerkin projection in a smaller vector space.
The idea is similar to that of interpolating 
   the boundary condition in the standard finite element method.
Due to the pointwise interpolation,  our method yields  a smaller system of equations and  a better condition number.
The number of unknowns on each element is reduced significantly 
   from $(k^2+3k+2)/2$ to $3k$ for the
   $P_k$ ($k\ge 3$) finite element. 
We construct 2D $P_2$ conforming and nonconforming, and $P_k$ ($k\ge3$)
     conforming interpolated Galerkin finite elements on triangular grids.
This interpolated Galerkin finite element method is proved to converge at the optimal order.
Numerical tests and comparisons with the standard finite elements
   are presented, verifying the theory and showing advantages of the interpolated Galerkin 
     finite element method. 
  \vskip 15pt


 \vskip 15pt


\end{abstract}
\maketitle

\section{Introduction}

 Standard  finite element methods use the full $P_k$ polynomials (of total degree $\leq k$) on each element
 (e.g. triangle or tetrahedron), in order to achieve the optimal order of approximation, 
   in solving partial differential equations.
 In certain situations the $P_k$ polynomial space is enriched by the so-called bubble functions,
  for stability or continuity,
   cf. \cite{Arnold,  Brenner-Scott, Falk, Hu-Huang-Zhang, Hu-Zhang, 
 Hu-Zhang-Low, Huang-Zhang, Schumaker, Zhang, Zhang-3D-C1, Zhang-4D, Zhang-P3}. 
But only in one case we use a proper subspace of $P_k$ polynomials while retaining the
   optimal order, $O(h^k)$ in $H^1$-norm, of convergence.
That is the harmonic finite element method for solving the Laplace equation, $\Delta p=p_{xx}+p_{yy}=0$,
   where only harmonic polynomials in $P_k$ are used \cite{Sorokina2, Sorokina}.
 
For example,  in the $P_2$ nonconforming element method for solving the following Laplace equation,
\an{ \label{h-e} \ad{ -\Delta u & =0, \quad \t{ in } \ \Omega, \\
                      u&=f, \quad \t{ on } \ \partial \Omega, } }
            where $\Omega$ is a  bounded polygonal domain in $\RR^2$,
  the five basis functions on the element boundary are harmonic polynomials and only the
   sixth basis function (which vanishes on the 6 Gauss-Legendre points on the three edges)
  is not a harmonic polynomial.
So, in \cite{Sorokina}, the 6th basis function of the $P_2$ nonconforming element  
    is thrown away, on every triangle, in the harmonic finite element method.
For example, on a uniform triangular grid on a square domain, the
  number of unknowns is reduced from $(2n-1)^2+2n^2$ to $(2n-1)^2$, about one-third less.
But the harmonic finite element method cannot be applied directly to the Poisson equation,
\an{ \label{e} \ad{ -\Delta u & =f, \quad \t{ in } \ \Omega, \\
                      u&=0, \quad \t{ on } \ \partial \Omega, } }
            where $\Omega$ is a  bounded polygonal domain in $\RR^2$.
The sixth basis function of $P_2$ nonconforming finite element must be added to the
  harmonic finite element method.
This is then the standard $P_2$ nonconforming element method where the solution is
\an{\label{2s} u_h=\sum_{\b x_i\in \partial K\setminus \partial \Omega} u_i \phi_i + 
         \sum_{\b x_j\in K^o} u_j \phi_j +
         \sum_{\b x_k\in \partial \Omega} c_k \phi_k, }
where $c_k$ 
    are interpolated values on the boundary, and $u_i$ and $u_j$
    are obtained from the Galerkin projection (from the solution of a discrete linear system of
  equations).
But the sixth basis function is local and the only non-harmonic polynomial
   which can be obtained from the right hand side function $f$
   in \eqref{e}.  That is,  the solution of
   the $P_2$ nonconforming interpolated Galerkin finite element is  
\an{\label{1s} u_h=\sum_{\b x_i\in \partial K\setminus \partial \Omega} u_i \phi_i + 
         \sum_{\b x_j\in K^o} c_j \phi_j +
         \sum_{\b x_k\in \partial \Omega} c_k \phi_k, }
where $c_j$ (could be $f(\b x_j)$ depending on which $\phi_j$ is used)
   are interpolated values of the right hand side function $f$,
   $c_k$ are interpolated boundary values, and only $u_i$ 
    are obtained from the Galerkin projection.
The new method does not only reduce the number of unknowns (from $O(k^2)$ to $O(k)$),
   but also improves the condition number.  It is totally different from the
  traditional finite element static condensation which does Gaussian elimination from internal
   degrees of freedom first.

In this work,  in addition to constructing special $P_2$ conforming and nonconforming
    interpolated finite elements, we redefine the basis functions of 
   the $P_k$ ($k\ge 3$) Lagrange finite element.
We keep the Lagrange nodal values on the boundary of each element, and replace
  the internal Lagrange nodal values by the internal Laplacian values at these
  internal Lagrange nodes.
This way, the linear system of Galerkin projection equations involves only the
  unknowns on the inter-element boundary.
Therefore the number of unknowns on each element is reduced from $(k+1)(k+2)/2$ to $3k$
  as all internal unknowns are interpolated   by the given function $f$ directly.
We show that the interpolated Galerkin finite element solution converges at the optimal order.
Numerical tests are provided to the above mentioned interpolated Galerkin finite elements, 
   in comparison with the standard finite element method.

\section{The $P_2$ interpolated Galerkin conforming finite element}

The $P_2$ interpolated Galerkin conforming
   finite element is defined only on macro-element grids, while the 
  rest higher order $P_k$ elements are defined on general triangular grids. 
We will define another $P_2$ interpolated nonconforming element next section on general 
    triangular grids.

A $P_2$ harmonic polynomial is a linear combination of $1, x, y, x^2-y^2$ and $xy$.
A $P_2$ interpolated finite element basis function is a linear combination of
    $1, x, y, x^2-y^2, xy$ and $x^2+y^2$.
  Only the last basis function has a non-zero Laplacian.

Let the union of four triangles $\hat K=\cup_{i=1}^4 K_i$ be a reference macro-element
    shown in Figure \ref{K-hat} (left).
On the reference macro-element $\hat K$, the $P_2$ finite element space is
\an{\label{P2hat}  P_{\hat K}:&= \{ v_h \in L^2(\hat K) \mid 
     v_h|_{K_i}\in P_{2}; \  v_h\in C^0(\b x_i), \ i=1,2,3,4; \\
   \nonumber     &\qquad v_h\in C^1(\b x_9); \ \Delta v_h\in P_0 \}, }
where $v_h \in C^0(\b x_i)$ means that the two adjoining at $\b x_i$  polynomial pieces 
   of $v_h$ have the same value at $\b x_i$, $v_h \in C^1(\b x_9)$ means that the four adjoining at $\b x_9$ polynomial pieces of $v_h$  and its first derivatives have matching values at $\b x_9$, and $\Delta v_h\in P_0$ means that the four adjoining at $\b x_9$ polynomial pieces of $v_h$ have 
   a matching constant Laplacian. 
These conditions immediately imply that  $v_h$ is continuous on $\hat K$.
We will  show that the dimension of the  space $P_{\hat K}$ is 9, and each such $P_2$
  function is uniquely determined by its 8 nodal values,  $v_h(\b x_i)$, $i=1,2,\dots,8$, and the value of $\Delta v_h$, see
   Figure \ref{K-hat} (left).

\begin{figure}[h!]
\begin{minipage}{0.47\linewidth}
\begin{center} \setlength\unitlength{5pt}
\begin{picture}(20,20)(0,0) 
  \def\cr{\begin{picture}(20,20)(0,0)\put(0,0){\line(1,0){20}}\put(0,20){\line(1,0){20}}
   \put(0,0){\line(0,1){20}} \put(20,0){\line(0,1){20}} 
    \put(0,0){\line(1,1){20}} \put(20,0){\line(-1,1){20}}\end{picture}}
 \put(0,0){\cr} \put(9.5,3){$K_1$} \put(16,9.5){$K_2$} \put(9.5,16){$K_4$} \put(3,9.5){$K_3$} 
  \put(-2.5,0){$\b x_1$} \put(21,0){$\b x_2$} \put(21,20){$\b x_3$} \put(-2.5,20){$\b x_4$} 
  \put(11,9.2){$\b x_9$}
  \put(10,-1.5){$\b x_5$} \put(21,9){$\b x_6$}\put(10,20.5){$\b x_7$} \put(-2.5,9){$\b x_8$}
  \put(0,0){\circle*{0.5}}\put(10,0){\circle*{0.5}}\put(20,0){\circle*{0.5}}
  \put(0,10){\circle*{0.5}}                      \put(20,10){\circle*{0.5}}
  \put(0,20){\circle*{0.5}}\put(10,20){\circle*{0.5}}\put(20,20){\circle*{0.5}}
 \end{picture}\end{center}
\end{minipage}
\begin{minipage}{0.47\linewidth}
\begin{center} \setlength\unitlength{5pt}
\begin{picture}(20,20)(0,0) 
  \def\cr{\begin{picture}(20,20)(0,0)\put(0,0){\line(1,0){20}}\put(0,20){\line(1,0){20}}
   \put(0,0){\line(0,1){20}} \put(20,0){\line(0,1){20}} 
    \put(0,0){\line(1,1){20}} \put(20,0){\line(-1,1){20}}\end{picture}}
 \put(0,0){\cr} \put(3.5,13.3){$c_{13}$} \put(15.5,5.5){$c_{11}$} \put(15.5,13.7){$c_{12}$} \put(3.2,6.1){$c_{10}$} 
  \put(-2.5,0){$ c_1$} \put(21,0){$ c_2$} \put(21,20){$ c_3$} \put(-2.5,20){$ c_4$} 
  \put(11,9.2){$ c_9$}
  \put(10,-1.5){$ c_5$} \put(21,9){$ c_6$}\put(10,20.5){$ c_7$} \put(-2.5,9){$ c_8$}
  \put(0,0){\circle*{0.5}}\put(10,0){\circle*{0.5}}\put(20,0){\circle*{0.5}}
  \put(0,10){\circle*{0.5}}                      \put(20,10){\circle*{0.5}}
  \put(0,20){\circle*{0.5}}\put(10,20){\circle*{0.5}}\put(20,20){\circle*{0.5}}\put(5,5){\circle*{0.5}}\put(5,15){\circle*{0.5}}\put(15,5){\circle*{0.5}}\put(15,15){\circle*{0.5}}\put(10,10){\circle*{0.5}}
 \end{picture}\end{center}
\end{minipage}
\caption{\label{K-hat} The reference macro-element,  $\hat K=\cup_{i=1}^4 K_i=[-1,1]^2$ (left), and the B-coefficients associated with the domain points in $\hat K$ (right)}
\end{figure}
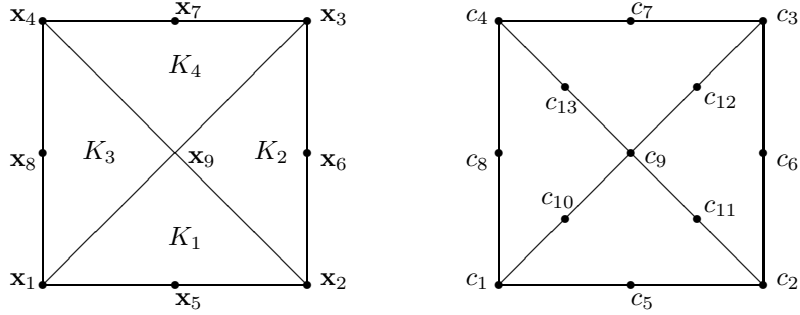


\begin{theorem} Consider $S^0_2 (\hat K)\xrightarrow[]{\Delta} S^{-1}_0(\hat K)$, where  
\a{  S^{0,1}_2(\hat K) &: = \{ s|_{K_i}\in P_{2},  \ i=1,2,3,4; \ 
     s\in C^1(\b x_9); \  s\in C^0(\hat K) \},\\
   S^{-1}_0(\hat K) &: = \{ s|_{K_i}\equiv c_i\in\RR,  \quad i=1,2,3,4. \}.}
Then the image of $S^0_2 (\hat K)$ is a three dimensional subspace of $S^{-1}_0(\hat K)$ consisting of piecewise constants satisfying 
the condition $c_1-c_2+c_3-c_4=0$.
\end{theorem}
\begin{lemma} The conforming $P_2$  finite element function 
    $u_h\in P_{\hat K_h}$ is unisolvent by the eight nodal values,  $u_h(\b x_i)$,
   $i=1,2,\dots,8$, and the value $\Delta u_h(\b x_9)$.
\end{lemma}

\begin{proof}  We use the Bernstein-B\'ezier form of $u_h$ on $\hat K_h$ with the B-coefficients $c_1$, $c_2$, \dots, $c_{13}$ associated 
with the domain points in $\hat K_h$ as depicted in Figure~\ref{K-hat} (right), see e.g. Chapter 2 of~\cite{LaiSch} for relevant definitions. Using the 
eight nodal values,  $u_h(\b x_i)$, $i=1,2,\dots,8$, we compute the eight B-coefficients by interpolation conditions as follows
 \begin{alignat*}{2}
 &c_i=u_h(\b x_i), &i=1,\dots,4,\\
 &c_5=2u_h(\b x_5)-\left(c_1+c_2\right)/2,\hskip 10pt &c_6=2u_h(\b x_6)-\left(c_2+c_3\right)/2, \\
 &c_7=2u_h(\b x_7)-\left(c_3+c_4\right)/2,\hskip 10pt &c_8=2u_h(\b x_6)-\left(c_4+c_1\right)/2. \\
 \end{alignat*}
By Lemma 4.1 in~\cite{va1}, the following four conditions are necessary and sufficient for $\Delta u_h=\Delta u_h(\b x_9)$  on each triangle $K_i$, $i=1,\dots,4$,  in the  square $\hat K_h$:
\begin{align}\label{system1} \ad{
&2c_9+c_1+c_2-2c_{10}-2c_{11}=\Delta u_h(\b x_9)/2,\\ 
&2c_9+c_2+c_3-2c_{11}-2c_{12}=\Delta u_h(\b x_9)/2,\\ 
&2c_9+c_3+c_4-2c_{12}-2c_{13}=\Delta u_h(\b x_9)/2,\\ 
&2c_9+c_4+c_1-2c_{13}-2c_{10}=\Delta u_h(\b x_9)/2. }
\end{align}
By Theorem 2.28 in~\cite{LaiSch}, the following two conditions  are necessary and sufficient for $u_h$ to be $C^1$  at the center $\b x_9$ 
of the  square $\hat K_h$:
\begin{align}\label{system2} \ad{
&2c_9-c_{10}-c_{12}=0,\\ 
&2c_9-c_{11}-c_{13}=0. }
\end{align}
Note that the alternating sum of the four equations in~(\ref{system1}) vanishes. Thus we only consider the first three equations of~(\ref{system1}).  Substituting $c_{13}=c_{10}+c_{12}-c_{11}$, and $2c_9=c_{10}+c_{12}$  from~(\ref{system2}), into~(\ref{system1}), we obtain a system of three equations with three unknowns that has a unique solution given by
\begin{align}\label{solution} \ad{
&c_9=(c_1+c_2+c_3+c_4-2\Delta u_h(\b x_9))/4,\\ 
&c_{10}=(2c_1+c_2+c_4-\Delta u_h(\b x_9))/4,\\ 
&c_{11}=(2c_2+c_1+c_3-\Delta u_h(\b x_9))/4,\\ 
&c_{12}=(2c_3+c_2+c_4-\Delta u_h(\b x_9))/4,\\ 
&c_{13}=(2c_4+c_1+c_3-\Delta u_h(\b x_9))/4. }
\end{align}
Therefore, all thirteen B-coefficients $c_1,c_2,\dots, c_{13}$ have been uniquely determined by the eight nodal values $u_h(\b x_i)$,
   $i=1,2,\dots,8$, and by $\Delta u_h(\b x_9)$.
\end{proof}

Let $\mathcal{M}_h=\{K :  \cup K=\Omega \}$ be a square subdivision of the domain $\Omega$.
We subdivide each rectangle $K$ in to four triangles $K_i$ as in Figure \ref{K-hat}, and let
   $\mathcal{T}_h=
   \{K_i : K_i \subset K \} $ be the corresponding triangular grid of grid-size $h$.
The $P_2$ finite element space on the grid is defined by
\an{\label{P2N} 
   V_h & =\{ v_h \in H^1_0(\Omega)
          \mid \ v_h|_K =\sum_{i=1}^8 c_i \phi_i + c_9 \phi_9  \in P_{K}
           \ \forall K\in\mathcal{M}_h \},  }
  where $P_K$ is defined in \eqref{P2hat}, 
    basis $\phi_i(\b x_j)=\delta_{ij}$ and $\Delta \phi_{i9}(\b x_9)=\delta_{i9}$.  
The interpolated Galerkin finite element problem reads:   Find 
 $u_h=\sum_{K\in\mathcal{M}_h}\Big(\sum_{i=1}^8 u_i \phi_i - f(\b x_9) \phi_9\Big)$
  such that
\an{ \label{E-2}
  (\nabla u_h, \nabla v_h)=(f,v_h) \quad\forall v_h=\sum_{K\in\mathcal{M}_h} \sum_{i=1}^8 v_i \phi_i. }

\section{The $P_2$ interpolated nonconforming finite element}

We define a  $P_2$ interpolated Galerkin nonconforming finite element on
  general triangular grids in this section.
This element is the best one to describe the difference between  interpolated Galerkin
   finite element methods and standard Galerkin finite element methods.
The $P_2$ nonconforming finite element function is continuous on the two Gauss-Legendre 
  points of every edge.
But the set of 6 nodal values of a $6$-dimensional $P_2$ polynomial is linearly
  dependent.
We can use the 5-dimensional harmonic $P_2$ polynomials to build these 5 basis 
   functions which has
  non-zero values at the 6 Gauss-Legendre points and zero Laplacian at the 
  barycenter of triangle.
As these 6 Gauss-Legendre points on edges are always on an ellipse, the last $P_2$ basis function 
   has a constant Laplacian 1 everywhere on the triangle 
   and vanishes at the Gauss-Legendre points.
This is how the basis functions are defined 
  in the standard $P_2$ nonconforming finite element.
Now, instead of solving the coefficient of this last basis function from the discrete equations,  
  we can interpolate the right hand side function  to get this coefficient directly.

Let $\mathcal {T}_h$ be a shape-regular, quasi-uniform triangulation on $\Omega$.
The $P_2$ non-conforming interpolated finite element space is defined by
\an{\label{P2N} \ad{
   V_h =\{ v_h \in L^2(\Omega)
          \mid \ &  v_h \t{ is continuous at two Gauss points each edge }, \\
         &  v_h \t{ is zero at two Gauss points on boundary edge }, \\
         & v_h|_K =\sum_{i=1}^5 c_i \phi_i + c_0 \phi_0  \in P_2(K) 
           \ \forall K\in\mathcal{T}_h \},  } }
  where  $\{\phi_i, i=1,...,5\}$ 
    are global basis functions restricted on $K$ which are
   $P_2$ harmonic functions (i.e., spanned by $\{1,x,y,xy,x^2-y^2\}$, cf. \cite{Sorokina})
    and $ \Delta \phi_{0}(\b x_0)=-1$,  $\phi_{0}(\b x_i)=0$, $i=1,...,5$,
      cf. Figure \ref{f-p2}.

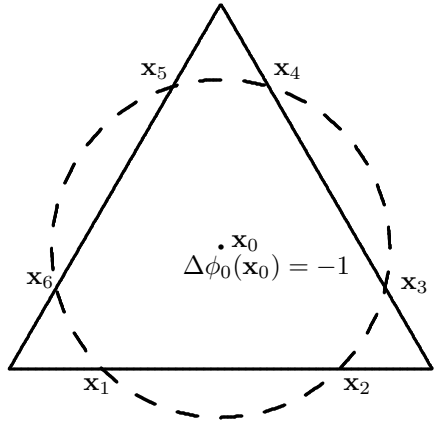
\begin{figure}[ht] \setlength{\unitlength}{0.8pt}
 \begin{center}\begin{picture}(  200.,  192)(  0.,  -20.)
     \def\lb{\circle*{0.8}}\def\lc{\vrule width1.2pt height1.2pt}
     \def\la{\circle*{0.3}}

     \multiput(   0.00,   0.00)(   0.125,   0.216){799}{\la}
     \multiput(   0.00,   0.00)(   0.250,   0.000){800}{\la}
     \multiput( 200.00,   0.00)(  -0.125,   0.216){799}{\la}
  \multiput( 180.00,  57.00)( -0.011,  0.166){ 62}{\la}
  \multiput( 177.28,  77.70)( -0.054,  0.158){ 62}{\la}
  \multiput( 169.29,  96.98)( -0.093,  0.139){ 62}{\la}
  \multiput( 156.59, 113.55)( -0.125,  0.110){ 62}{\la}
  \multiput( 140.03, 126.26)( -0.149,  0.074){ 62}{\la}
  \multiput( 120.75, 134.26)( -0.163,  0.033){ 62}{\la}
  \multiput( 100.05, 137.00)( -0.166, -0.011){ 62}{\la}
  \multiput(  79.35, 134.29)( -0.158, -0.053){ 62}{\la}
  \multiput(  60.06, 126.32)( -0.139, -0.092){ 62}{\la}
  \multiput(  43.49, 113.62)( -0.110, -0.125){ 62}{\la}
  \multiput(  30.76,  97.08)( -0.074, -0.149){ 62}{\la}
  \multiput(  22.75,  77.80)( -0.033, -0.163){ 62}{\la}
  \multiput(  20.00,  57.10)(  0.011, -0.166){ 62}{\la}
  \multiput(  22.70,  36.40)(  0.053, -0.158){ 62}{\la}
  \multiput(  30.66,  17.11)(  0.092, -0.139){ 62}{\la}
  \multiput(  43.34,   0.52)(  0.125, -0.110){ 62}{\la}
  \multiput(  59.88, -12.21)(  0.149, -0.074){ 62}{\la}
  \multiput(  79.15, -20.24)(  0.163, -0.033){ 62}{\la}
  \multiput(  99.84, -23.00)(  0.166,  0.011){ 62}{\la}
  \multiput( 120.55, -20.32)(  0.158,  0.053){ 62}{\la}
  \multiput( 139.85, -12.37)(  0.139,  0.092){ 62}{\la}
  \multiput( 156.44,   0.30)(  0.110,  0.125){ 62}{\la}
  \multiput( 169.19,  16.83)(  0.074,  0.149){ 62}{\la}
  \multiput( 177.22,  36.10)(  0.033,  0.163){ 62}{\la}
  \put(35,-10){$\b x_1$}\put(158,-10){$\b x_2$}
  \put(8,40){$\b x_6$}\put(185,38){$\b x_3$}
  \put(62,139){$\b x_5$}\put(125,139){$\b x_4$}
  \put(105,57.6){$\b x_0$}\put(100,57.6){\circle*{3}}
  \put(82,45){$\Delta \phi_0(\b x_0)=-1$}
   
 \end{picture}\end{center}
\caption{Nodal points of $P_2$ non-conforming finite elements. \label{f-p2} }
\end{figure}

The $P_2$-nonconforming, interpolated Galerkin finite element problem reads:   Find 
 $u_h=\sum_{K\in\mathcal{T}_h}\Big(\sum_{i=1}^5 u_i \phi_i + f(\b x_0) \phi_0\Big)$
  such that
\an{\label{E-2n} (\nabla_h u_h, \nabla_h v_h)=(f,v_h) \quad\forall v_h=\sum_{K\in\mathcal{T}_h} \sum_{i=1}^5 v_i \phi_i. }

\section{$P_k$ ($k\ge 3$) interpolated Galerkin finite elements}
There is no $P_1$ interpolated finite element as the Laplacian of a $P_1$ polynomial is
   zero.
It needs special cares for $P_2$ interpolated finite elements, as we did in the last two
   sections.
But for $P_3$ and above interpolated finite elements, we can simply replace all
   internal Lagrange degrees of freedom by the Laplacian values, which are
  obtained from the given right hand side function $f$.
However we could not prove the uni-solvence for general $k$.
We define another type interpolated finite element for $P_k$ ($k\ge 4$), where
  we use local averaging Laplacian values in stead of pointwise Laplacian value.

Let $\mathcal {T}_h$ be a shape-regular, general quasi-uniform triangulation on $\Omega$.
The $P_3$  interpolated finite element space is defined by
\an{\label{P-3} \ad{
   V_h =\{ v_h \in H^1_0(\Omega)
          \mid \  
         & v_h|_K =\sum_{i=1}^{9} c_i \phi_i +  c_0 \phi_0  \in P_3(K) 
           \ \forall K\in\mathcal{T}_h \},  } }
  where $\{\phi_i, i=1,...,9\}$ 
    are boundary Lagrange basis functions which have vanishing Laplacian at 
    the barycenter $\b x_0$ of $K$,  and
   $\phi_0$ vanishes on the three edges of $K$ and  $\Delta \phi_{0}(\b x_0)=-1$,
      cf. Figure \ref{f-p3}.

\begin{figure}[ht] \setlength{\unitlength}{1pt}
 \begin{center}\begin{picture}(160.,  100)(  0.,  0)

\put(0,0){\line(1,0){160}} \put(0,0){\line(4,5){80}} \put(160,0){\line(-4,5){80}} 

    \multiput(0, 0)(26.6,33.3){4}{\circle*{5}}\multiput(160, 0)(-26.6,33.3){3}{\circle*{5}}  
   \multiput(0, 0)(53.3,0){4}{\circle*{5}} \put(113,70){$\phi_i(\b x_l)=\delta_{i,l}$}
     \put(80,33.3){\circle*{5}}
    \put(42,40){$\Delta\phi_0( \b x_0)=-1$}
 \end{picture}\end{center}
\caption{Nodal points of $P_3$ interpolated finite elements. \label{f-p3} }
\end{figure}
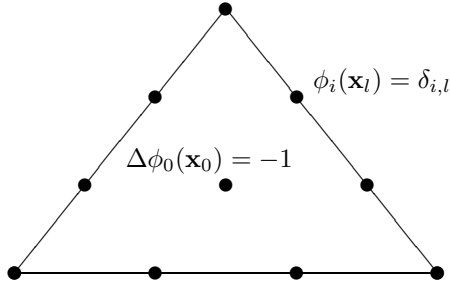

\begin{lemma} The $(9+1)$ nodal degrees of freedom in \eqref{P-3} 
   uniquely define a $P_3$ polynomial.
\end{lemma}

\begin{proof} We have a square  system of 9 linear equations with 9 unknowns.
The uniqueness guarantees existence.  Let $u_h$ be a solution of the homogeneous 
   system. Then $u_h$ vanishes on 3 edges, cf. Figure \ref{P-3}.  $u_h=C b$, where 
  $b$ is the $P_3$ bubble function on $K$, vanishing on the edges and assuming value
   $1$ at the barycenter $\b x_0$.
  Since $\Delta u_h(\b x_0)=0$ and  $\Delta u_h$ is a linear function,
   we have, by symmetry, 
\a{ 0 &=-\Delta u_h(\b x_0)b =-\int_K (\Delta u_h)b \; d\b x 
    \\ &  =\int_K \nabla u_h \cdot \nabla b d \b x = C \int_K |\nabla  b|^2 d \b x. }
 Thus $C=0$, $u_h=0$ and the lemma is proved.
\end{proof}

The $P_3$  interpolated Galerkin finite element problem reads: 
\an{\nonumber \t{  Find } \
   u_h&=\sum_{K\in\mathcal{T}_h}\Big(\sum_{i=1}^{9} u_i \phi_i + 
          f(\b x_0)_K \phi_0\Big) \ \t{
  such that } \\
   \label{E-3} (\nabla u_h, \nabla v_h)&=(f,v_h) \quad\forall 
              v_h=\sum_{K\in\mathcal{T}_h} \sum_{i=1}^{9} v_i \phi_i. }

For defining $P_k$ ($k\ge 4$) interpolated finite elements,  we explicitly define
  two types degrees of freedom.   
Let boundary nodal-value linear functional $F_i$ be
\an{\label{F-i} F_i(u) = u(\b x_i),  \quad i=1,...,3k, \t{ \ cf. Figure \ref{f-p5}.}  }
Let the Laplacian moment linear functional $G_j$ be
\an{\label{G-j} G_j(u) =\int_K p_j b \Delta u \, d\b x, \quad j=1,...,d_{k-3},  }
where $p$ is again the cubic bubble function on $K$,   $d_{k-3}=\dim P_{k-3}$,
  and $\{ p_j\}$ is an orthonormal basis obtained by the Gram-Schmidt process on 
  $\{ p_j=1,x,y,x^2,..., y^{k-3}\}$ under the inner product 
\a{ (u,v)_G = \int_K \nabla(b u)\cdot \nabla(b v) d\b x. }

\begin{figure}[ht] \setlength{\unitlength}{1pt}
 \begin{center}\begin{picture}(160.,  100)(  0.,  0)

\put(0,0){\line(1,0){160}} \put(0,0){\line(4,5){80}} \put(160,0){\line(-4,5){80}} 

     \multiput(0, 0)(16,20){6}{\circle*{5}}\multiput(160, 0)(-16,20){6}{\circle*{5}}  
   \multiput(0, 0)(32,0){6}{\circle*{5}} \put(103,78){$F_i(u)=u(\b x_i)$}
  \put(115,64){$ \b x_i$}   \put(35,27){$G_j(u)=\int_K p_jb\Delta u\,d\b x$}
 \end{picture}\end{center}
\caption{Nodal points and moments of $P_5$ interpolated finite elements. \label{f-p5} }
\end{figure}
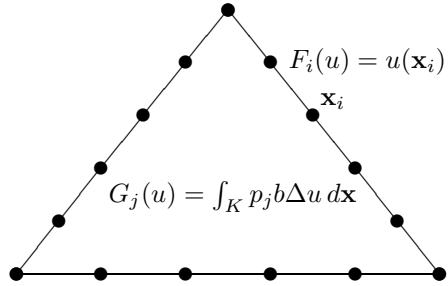

\begin{lemma}\label{l-k} The $(3k+d_{k-3})$ linear functionals in \eqref{F-i} and \eqref{G-j}
   uniquely define a $P_k$ polynomial.
\end{lemma}
\begin{proof} Because $\dim P_k=3k+d_{k-3}$,  we have a square linear system
   of equations when applying the functionals to determine a $P_k$ polynomial.
 We only need to show the uniqueness. 
  
Let $u\in P_k$ such that $F_i(u)=0$ and $G_j(u)=0$,  $i=1,...,3k, 
   j=1,..., d_{k-3}$.
Because $u=0$ on the three edges (cf. Figure \ref{f-p5}) we have 
\a{  u=bp \t{ \ for some } p\in P_{k-3}. } 
Let the combination of $p_j$ in $G_j$, defined in \eqref{G-j}, be $p$.
We get
\a{ 0&=\sum_{j=0}^{d_{k-3}} c_j G_j(u) = \int_K \sum_{j=0}^{d_{k-3}} c_j p_j b \Delta u\, d\b x\\
   &=  \int_K  p b \Delta u\, d\b x=  \int_K  |\nabla u|^2\, d\b x. }
Thus $\nabla u=0$ and $u=C$.  Because $u=bp=0$ on the boundary,  $u=C=0$.  
The proof is completed. 
\end{proof}
 
The $P_k$ ($k\ge 4$) interpolated finite element space is defined by
\an{\label{P-k} \ad{
   V_h =\{ v_h \in H^1_0(\Omega)
          \mid \  
         & v_h|_K =\sum_{i=1}^{3k} c_i \phi_i + \sum_{j=1}^{d_{k-3}} c_j \psi_j  \in P_k(K) 
           \ \forall K\in\mathcal{T}_h \},  } }
  where $\{\phi_i, \psi_j\}$ is the dual basis of $\{F_i, G_j\}$, by Lemma \ref{l-k}. 
The $P_k$ ($k\ge 4$) interpolated Galerkin finite element problem reads: 
\an{\label{u-d} \t{  Find } \
   u_h&=\sum_{K\in\mathcal{T}_h}\Big(\sum_{i=1}^{3k} u_i \phi_i - 
        \sum_{j=1}^{d_{k-3}} (f,\psi_j)_K \psi_j\Big) \ \t{
  such that } \\
   \label{E-k} (\nabla u_h, \nabla v_h)&=(f,v_h) \quad\forall 
              v_h=\sum_{K\in\mathcal{T}_h} \sum_{i=1}^{3k} v_i \phi_i. }

\section{Convergence theory} 

\begin{theorem}\label{m1}
 Let $u$ and $u_h$ be the exact solution of \eqref{e} and the $P_k$ ($k\ge 4$)
  interpolated finite element solution 
   of \eqref{E-k}, respectively.
  Then
   \an{\label{c11} \| u-u_h\|_{0}+h | u-u_h |_{1} \le C h^{k+1} | u |_{k+1},  } 
where $|\cdot|_k$ is the Sobolev (semi-)norm $H^k(\Omega)$.
\end{theorem}

\begin{proof}  Testing \eqref{e} by $v_h=\phi_i
    \in H^1_0(\Omega)$,  we have
\an{\label{e-v} (\nabla u, \nabla v_h) &=(f,v_h).  }
Subtracting \eqref{E-k} from \eqref{e-v},
\an{\label{o1} (\nabla(u-u_h), \nabla v_h) &= 0 . }
Testing \eqref{e} by $v_h=\psi_j\in H^1_0(\Omega)$, by \eqref{G-j} and \eqref{u-d},
we get
\an{ \label{o2} \ad{ 
          (\nabla(u-u_h), \nabla \psi_j)  
    &=-\int_K \Delta u \psi_j d \b x - (f,\psi_j)_K\int_K |\nabla \psi_j|^2 d \b x \\
              &= \int_K  f\psi_j d \b x  -(f,\psi_j)_K
               = 0. } }
Combining \eqref{o1} and \eqref{o2} implies
\a{   |u-u_h |_{1}^2 &= (\nabla(u-u_h),\nabla(u-I_h u))\\
                     &\le |u-u_h |_{1} |u-I_h u |_{1}
              \le |u-u_h |_{1}  Ch^k \|u\|_{k+1}, }
  where $I_h$ is the interpolation operator to $V_h$.
This completes the $H^1$ estimate.   Let $w\in H^2(\Omega) \cap H^1_0(\Omega)$ solve
\a{ (\nabla w,\nabla v) = ( u-u_h, v) \quad \forall v\in V_h.  }
We assume $H^2$ regularity for the solution, i.e., 
\a{ \|w\|_2 \le C \|u-u_h\|_0.  }
Let $w_h$ be $P_k$ interpolated finite solution of $w$.
Then \a{ \|u-u_h\|_0^2 &= (\nabla w, \nabla (u-u_h)) = (\nabla( w-w_h), \nabla (u-u_h)) \\
      &\le |w-w_h|_1 |u-u_h| \le C h |w|_2 C h^k  \|u\|_{k+1}
     \\&  \le \|u-u_h\|_0 C h^{k+1}  \|u\|_{k+1}.  }
This gives the $L^2$ error estimate.
 \end{proof}

\begin{theorem}\label{m2}
 Let $u$ be the exact solution of \eqref{e}.  Let $u_h$ the $P_2$ conforming,
    or the $P_2$ nonconforming,  or the $P_3$ finite element solution 
   of \eqref{E-2}, \eqref{E-2n}, or \eqref{E-3}, respectively.
  Then
   \an{\label{c11} \| u-u_h\|_{0}+h | u-u_h |_{1} \le C h^{k+1} | u |_{k+1},  } 
where $k=2$, or $3$.
\end{theorem}

\begin{proof} As there is one local/internal Laplacian degree of freedom,  the
   proof becomes very simple.  Testing \eqref{e} by nodal value basis $\tilde v_h=\phi_i$, 
     we have
\an{\label{e-v2} (\nabla u, \nabla \tilde v_h) &=(f,v_h).  }
Subtracting finite element equations from \eqref{e-v2},
\a{ (\nabla(u-\tilde u_h-u_0 ), \nabla v_h) &= 0, }
where we separate the finite element solution $u_h$ in to two parts,
   the nodal basis span part $\tilde u_h$ and the interpolated Laplacian part $u_0$
   (spanned by the last basis function $\phi_0$.)
Thus 
\a{   |u-u_h |_{1}^2 &= (\nabla(u-\tilde u_h-u_0),\nabla(u-\tilde v_h - u_0 ))\\
                     &= (\nabla(u-\tilde u_h-u_0),\nabla(u- I_h u ))\\
                     &\le |u-u_h |_{1} |u-I_h u |_{1}
              \le |u-u_h |_{1}  Ch^k \|u\|_{k+1}, }
  where $I_h$ is the interpolation operator to $V_h$.
This completes the $H^1$ estimate.   The $L^2$ error estimate is identical to above
   proof for Theorem \ref{m1}.
The treatment for the inconsistency by $P_2$ nonconforming element is standard, cf.  
 \cite{Li, Sorokina, Wang, ZhangMin, Zhang-Jump}.
 \end{proof}

\section{Numerical tests}
Let the  domain of the boundary value problem~\eqref{e} be 
  $\Omega=(0,1)^2$.  The exact solution is $u(x,y)=\sin \pi x \sin \pi y$.
  We chose a family of uniform triangular grids, shown in Figure~\ref{grids},  
     in all numerical tests on
    $P_k$ interpolated Galerkin finite element methods.

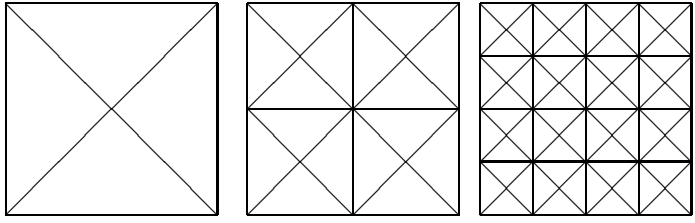
\begin{figure}[h!]
\begin{center} \setlength\unitlength{4pt}
\begin{picture}(70,20)(0,0)
  \def\tr{\begin{picture}(20,20)(0,0)\put(0,0){\line(1,0){20}}\put(0,20){\line(1,0){20}}
          \put(0,0){\line(0,1){20}} \put(20,0){\line(0,1){20}}  \put(20,0){\line(-1,1){20}} 
          \put(0,0){\line(1,1){20}} \end{picture}}
 \def\sq{\begin{picture}(20,20)(0,0)\put(0,0){\line(1,0){20}}\put(0,20){\line(1,0){20}}
          \put(0,0){\line(0,1){20}} \put(20,0){\line(0,1){20}}  \end{picture}}
 \def\pt{\begin{picture}(40,40)(0,0)\put(0,0){\line(1,0){40}}\put(0,40){\line(1,0){40}}
          \put(0,0){\line(0,1){40}} \put(40,0){\line(0,1){40}}
         \put(0,20){\line(1,0){10}}  \put(30,20){\line(1,0){10}}
         \put(20,0){\line(0,1){10}} \put(20,30){\line(0,1){10}}
        \put(10,20){\line(1,1){10}}\put(10,20){\line(1,-1){10}}
        \put(30,20){\line(-1,1){10}}\put(30,20){\line(-1,-1){10}}         \end{picture}}
  \def\hx{\begin{picture}(20,20)(0,0)\put(0,0){\line(1,0){20}}\put(0,20){\line(1,0){20}}
    \put(0,10){\line(1,1){10}}
          \put(0,0){\line(0,1){20}} \put(20,0){\line(0,1){20}}  \put(10,0){\line(1,1){10}}\end{picture}}

  \multiput(0,0)(20,0){1}{\multiput(0,0)(0,20){1}{\tr}} 

\put(22,0){  \setlength\unitlength{2pt}\begin{picture}(20,20)(0,0)
  \multiput(0,0)(20,0){2}{\multiput(0,0)(0,20){2}{\tr}} \end{picture} }

\put(44,0){  \setlength\unitlength{1pt}\begin{picture}(20,20)(0,0)
  \multiput(0,0)(20,0){4}{\multiput(0,0)(0,20){4}{\tr}} \end{picture} }

 \end{picture}\end{center}
    \caption{The first three levels of grids in all numerical tests.  }
    \label{grids}
\end{figure}

We solve problem \eqref{e} first by the $P_2$ interpolated Galerkin 
   conforming finite element method 
  defined in \eqref{E-2} and by 
   the $P_2$ Lagrange finite element method, on same grids.
The errors and the orders of convergence are listed in Table \ref{t1}.
Both elements converge at the optimal order.

 \begin{table}[ht]
  \caption{\lab{t1} The error $e_h= I_h u- u_h$
     and the order of convergence, by the  $P_2$ conforming interpolated
      finite element and
    by the $P_2$ Lagrange finite element.  }
\begin{center}  \begin{tabular}{c|rr|rr|rr|rr}  
\hline  
grid  & $ \|e_h\|_{0}$ &$h^n$ &$ |e_h|_{1}$ & $h^n$ & $ \|e_h\|_{0}$ &$h^n$ & $ |e_h|_{1}$ & $h^n$   \\ 
\hline 
 & \multicolumn{4}{c|}{$P_2$ Interpolated conforming FE} &
        \multicolumn{4}{c}{$P_2$ Lagrange element}  \\ \hline  
 4&  0.614E-03&  3.2&  0.499E-01&  2.0&  0.615E-03&  3.2&  0.500E-01&  2.0\\
 5&  0.723E-04&  3.1&  0.124E-01&  2.0&  0.723E-04&  3.1&  0.124E-01&  2.0\\
 6&  0.887E-05&  3.0&  0.309E-02&  2.0&  0.887E-05&  3.0&  0.309E-02&  2.0\\
 7&  0.110E-05&  3.0&  0.773E-03&  2.0&  0.110E-05&  3.0&  0.773E-03&  2.0\\
 8&  0.138E-06&  3.0&  0.193E-03&  2.0&  0.138E-06&  3.0&  0.193E-03&  2.0\\
 9&  0.172E-07&  3.0&  0.483E-04&  2.0&  0.172E-07&  3.0&  0.483E-04&  2.0\\ 
     \hline
\end{tabular}\end{center} \end{table}

Next we solve the test problem \eqref{e} again, 
    by the  $P_2$ interpolated non-conforming finite element method \eqref{E-2n} and by 
   the standard $P_2$ nonconforming finite element method.
The errors and the orders of convergence are listed in Table \ref{t2}.
Again, both methods converge in the optimal order.

 \begin{table}[ht]
  \caption{\lab{t2} The error $e_h= I_h u- u_h$
     and the order of convergence, by the  $P_2$ interpolated nonconforming 
       finite element and
    by the $P_2$ nonconforming finite element.  }
\begin{center}  \begin{tabular}{c|rr|rr|rr|rr}  
\hline  
grid  & $ \|e_h\|_{0}$ &$h^n$ &$ |e_h|_{1}$ & $h^n$ & $ \|e_h\|_{0}$ &$h^n$ & $ |e_h|_{1}$ & $h^n$   \\ 
\hline 
 & \multicolumn{4}{c|}{$P_2$ Interpolated nonconforming FE} &
        \multicolumn{4}{c}{$P_2$ nonconforming element}  \\ \hline  
 2&  0.503E-02&  3.8&  0.839E-01&  2.3&  0.124E-01&  3.2&  0.186E+00&  2.5\\
 3&  0.118E-02&  2.1&  0.363E-01&  1.2&  0.164E-02&  2.9&  0.495E-01&  1.9\\
 4&  0.181E-03&  2.7&  0.111E-01&  1.7&  0.208E-03&  3.0&  0.126E-01&  2.0\\
 5&  0.244E-04&  2.9&  0.298E-02&  1.9&  0.260E-04&  3.0&  0.315E-02&  2.0\\
 6&  0.316E-05&  3.0&  0.767E-03&  2.0&  0.325E-05&  3.0&  0.789E-03&  2.0\\
 7&  0.406E-06&  3.0&  0.194E-03&  2.0&  0.407E-06&  3.0&  0.197E-03&  2.0\\ 
     \hline
\end{tabular}\end{center} \end{table}

In Table \ref{t3} we list the results of $P_3$ interpolated finite elements \eqref{E-3}
  and  
   the $P_3$ Lagrange finite elements. 

\begin{table}[h!]
  \caption{\lab{t3} The error $e_h= I_h u- u_h$
     and the order of convergence, by the  $P_3$ interpolated finite element and
    by the $P_3$ Lagrange finite element.  }
\begin{center}  \begin{tabular}{c|rr|rr|rr|rr}  
\hline  grid  & $ \|e_h\|_{0}$ &$h^n$ &
    $ |e_h|_{1}$ & $h^n$ & $ \|e_h\|_{0}$ &$h^n$ &
    $ |e_h|_{1}$ & $h^n$   \\ 
\hline & \multicolumn{4}{c|}{$P_3$ interpolated element} &
        \multicolumn{4}{c}{$P_3$ Lagrange element}  \\ \hline 
 4&  0.119E-04&  4.0&  0.114E-02&  2.9&  0.118E-04&  4.0&  0.114E-02&  3.0\\
 5&  0.742E-06&  4.0&  0.143E-03&  3.0&  0.742E-06&  4.0&  0.143E-03&  3.0\\
 6&  0.464E-07&  4.0&  0.180E-04&  3.0&  0.464E-07&  4.0&  0.180E-04&  3.0\\
 7&  0.290E-08&  4.0&  0.225E-05&  3.0&  0.290E-08&  4.0&  0.225E-05&  3.0\\
 8&  0.181E-09&  4.0&  0.281E-06&  3.0&  0.181E-09&  4.0&  0.281E-06&  3.0\\ 
     \hline
\end{tabular}\end{center} \end{table}
\vskip 10pt

We then solve problem \eqref{e} by the $P_4$/$P_5$/$P_6$ interpolated 
    finite element methods \eqref{E-k} and by 
   the $P_4$/$P_5$/$P_6$ Lagrange finite element methods.
The errors and the orders of convergence are listed in Table \ref{t4}.
The optimal order of convergence is achieved in all cases.
\vskip 10pt
\begin{table}[h!]
  \caption{\lab{t4} The error $e_h= I_h u- u_h$
     and the order of convergence, by the  $P_4$/$P_5$/$P_6$ interpolated finite elements and
    by the $P_4$/$P_5$/$P_6$ Lagrange finite elements.  }
\begin{center}  \begin{tabular}{c|rr|rr|rr|rr}  
\hline  grid  & $ \|e_h\|_{0}$ &$h^n$ &
    $ |e_h|_{1}$ & $h^n$ & $ \|e_h\|_{0}$ &$h^n$ &
    $ |e_h|_{1}$ & $h^n$   \\ 
\hline & \multicolumn{4}{c|}{$P_4$ interpolated element} &
        \multicolumn{4}{c}{$P_4$ Lagrange element}  \\ \hline  
 4&  0.136E-06&  4.7&  0.132E-04&  3.9&  0.159E-06&  5.0&  0.142E-04&  4.0\\
 5&  0.464E-08&  4.9&  0.859E-06&  3.9&  0.501E-08&  5.0&  0.891E-06&  4.0\\
 6&  0.151E-09&  4.9&  0.547E-07&  4.0&  0.157E-09&  5.0&  0.557E-07&  4.0\\ 
     \hline  
\hline & \multicolumn{4}{c|}{$P_5$ interpolated element} &
        \multicolumn{4}{c}{$P_5$ Lagrange element}  \\ \hline   
 3&  0.484E-06&  6.0&  0.501E-04&  4.9&  0.478E-06&  6.0&  0.499E-04&  4.9\\
 4&  0.755E-08&  6.0&  0.158E-05&  5.0&  0.754E-08&  6.0&  0.158E-05&  5.0\\
 5&  0.118E-09&  6.0&  0.495E-07&  5.0&  0.118E-09&  6.0&  0.495E-07&  5.0\\ 
     \hline
\hline & \multicolumn{4}{c|}{$P_6$ interpolated element} &
        \multicolumn{4}{c}{$P_6$ Lagrange element}  \\ \hline    
 2&  0.144E-05&  6.6&  0.639E-04&  5.6&  0.336E-06&  7.2&  0.164E-04&  6.3\\
 3&  0.122E-07&  6.9&  0.102E-05&  6.0&  0.276E-08&  6.9&  0.259E-06&  6.0\\
 4&  0.972E-10&  7.0&  0.160E-07&  6.0&  0.218E-10&  7.0&  0.406E-08&  6.0\\ 
     \hline  
\end{tabular}\end{center} \end{table}


\clearpage

\begin{thebibliography}{999}
  
\bibitem{va1}
P. Alfeld, and T. Sorokina, Linear Differential Operators on Bivariate  
Spline Spaces and Spline Vector Fields, BIT Numerical Mathematics, 56(1),  15-32, 2016. 

\bibitem{Arnold}  D. N. Arnold, D. Boffi and R. S. Falk, 
    Approximation by quadrilateral finite elements. 
    Math. Comp. 71 (2002), no. 239, 909-922. 

\bibitem{Bramble}J. H. Bramble and S. R. Hilbert, 
  Estimation of linear functionals on Sobolev spaces with
applications to Fourier transforms and spline interpolation, SIAM J. Numer. Anal., 7
(1970),  113--124.
 


\bibitem{Brenner-Scott} S. C. Brenner and L. R. Scott,
 The mathematical theory of finite element methods.
    Third edition. Texts in Applied Mathematics,
     15. Springer, New York, 2008.

\bibitem{Falk} R. S.  Falk, P.  Gatto and P. Monk, 
   Hexahedral H(div) and H(curl) finite elements.
  ESAIM Math. Model. Numer. Anal. 45 (2011), no. 1, 115-143.

 
\bibitem{Hu-Huang-Zhang} J. Hu, Y. Huang and S. Zhang,
   The lowest order
    differentiable finite element on rectangular grids,
     SIAM Num. Anal.  49 (2011), No 4, 1350--1368.

\bibitem{Hu-Zhang} J. Hu and S. Zhang,
  The minimal conforming $H^k$ finite element spaces on
   $R^n$ rectangular grids,
    Math. Comp.  84  (2015),  no. 292, 563--579.

 \bibitem{Hu-Zhang-Low}   J. Hu and S. Zhang, {
Finite element approximations of symmetric
   tensors on simplicial grids in $R^n$:
   the lower order case, }
 Math. Models Methods Appl. Sci.  26 (2016), no. 9, 1649--1669. 


\bibitem{Huang-Zhang} Y. Huang and S. Zhang,
  Supercloseness of the divergence-free finite element solutions on rectangular grids,
   Commun. Math. Stat. 1 (2013), no. 2, 143--162. 
 

\bibitem{LaiSch}
 M.-J.~Lai and L.~L.~Schumaker,
  Spline functions on triangulations,
 Cambridge University Press, Cambridge, 2007.

\bibitem{Li}
M. Li, S. Mao and S. Zhang, 
New error estimates of nonconforming mixed finite element methods for the Stokes problem,
   Math. Methods Appl. Sci. 37 (2014), no. 7, 937--951. 

\bibitem{Schumaker} L. L. Schumaker, T. Sorokina and A. J.  Worsey,
   A C1 quadratic trivariate macro-element space
defined over arbitrary tetrahedral partitions. J. Approx. Theory,
  158 (2009), No. 1,  126--142.

\bibitem{Scott-Zhang}
L.R. Scott and S. Zhang,
   Finite element interpolation of nonsmooth functions satisfying boundary conditions,
    Math. Comp. 54 (1990), 483--493.
  

\bibitem{Sorokina2}  T. Sorokina and S. Zhang,
Conforming harmonic finite elements on the Hsieh-Clough-Tocher split of a triangle,
    Int. J. Numer. Anal. Model.,  17 (2020) no. 1, 54--67.

\bibitem{Sorokina}  T. Sorokina and S. Zhang,
  Conforming and nonconforming harmonic finite elements, 
   Applicable Analysis,  Appl. Anal. 99 (2020), no. 4, 569--584.

\bibitem{Wang} 
C. Wang, S. Zhang, Shangyou and J. Chen, 
A unified mortar condition for nonconforming finite elements,
  J. Sci. Comput. 62 (2015), no. 1, 179--197. 


\bibitem{ZhangMin} 
M. Zhang and S. Zhang, 
A 3D conforming-nonconforming mixed finite element for solving symmetric stress Stokes equations, Int. J. Numer. Anal. Model. 14 (2017), no. 4-5, 730--743. 

\bibitem{Zhang}  S. Zhang, A C1-P2 finite element without nodal basis,  
     M2AN 42 (2008), 175--192.

\bibitem{Zhang-3D-C1} S. Zhang, 
    A family of 3D continuously differentiable finite elements
           on tetrahedral grids,  Applied Numerical Mathematics,  
          59 (2009), no. 1,   219--233.  
      
 \bibitem{Zhang-4D} S. Zhang, 
 A family of differentiable finite elements on simplicial grids in four space dimensions, (Chinese)
     Math. Numer. Sin. 38 (2016), no. 3, 309--324.         

 \bibitem{Zhang-Jump} S. Zhang, 
 Coefficient jump-independent approximation of the conforming and nonconforming finite element solutions, Adv. Appl. Math. Mech. 8 (2016), no. 5, 722--736. 
  \bibitem{Zhang-P3} S. Zhang, 
A P4 bubble enriched P3 divergence-free finite 
  element on triangular grids,   Comput. Math. Appl. 74 (2017), no. 11, 2710--2722.
 



\end{thebibliography}
\end{document}